\documentclass[12pt,oneside,openany,article]{memoir}
\usepackage{mempatch}

\nouppercaseheads 
\usepackage[amsthm,thmmarks,hyperref]{ntheorem}
\usepackage[noTeX]{mmap}
\usepackage{etex}

\usepackage[english]{babel}
\usepackage[leqno]{mathtools}

\usepackage{mathrsfs}
\usepackage{upgreek}
\usepackage[compress,square,comma,numbers]{natbib}
\usepackage{hypernat} 

\usepackage{sfmath}

\usepackage{microtype}

\usepackage{subfig}
\usepackage{wrapfig}
\usepackage{array}

\usepackage{graphicx,color}
\definecolor{gray}{gray}{0}
\pagecolor{white}

\usepackage{hyperxmp}
\usepackage{xr-hyper}
\usepackage{nameref}
\usepackage[pdftex,bookmarks,pdfnewwindow,plainpages=false,unicode]{hyperref}

\usepackage{bookmark}

\usepackage{enumitem}

\usepackage{amssymb} 

\hypersetup{
colorlinks=true,
linkcolor=black,
citecolor=black,
urlcolor=blue,
pdfauthor={Victor Ivrii},
pdftitle={Schroedinger Operator with Strong Magnetic Field: Propagation of singularities and sharper asymptotics},
pdfsubject={Sharp Spectral Asymptotics},
pdfkeywords={Microlocal Analysis,  Sharp Spectral Asymptotics, Schroedinger Operator, Strong Magnetic Field, Propagation},
bookmarksdepth={4}
}

\numberwithin{equation}{chapter}

\theoremstyle{plain}
\newtheorem{theorem}{Theorem}[chapter]

\theoremstyle{definition}

\newtheorem{Problem}[theorem]{Problem}
\theoremstyle{remark}
\newtheorem{remark}[theorem]{Remark}
\newtheorem{example}[theorem]{Example}
\newtheorem{problem}[theorem]{Problem}

\numberwithin{equation}{chapter}

\DeclareMathAlphabet{\mathpzc}{OT1}{pzc}{m}{it}

\newcommand{\cA}{\mathcal{A}}

 \newcommand{\sL}{\mathscr{L}}

\newcommand{\MW}{{\mathsf{MW}}}

\newcommand{\R}{{\mathsf{R}}}

\newcommand{\T}{{\mathsf{T}}}

\newcommand{\sU}{{\mathsf{U}}}

\newcommand{\const}{{\mathsf{const}}}

\newcommand{\bH}{{\mathbb{H}}}

\newcommand{\bQ}{{\mathbb{Q}}}
\newcommand{\bR}{{\mathbb{R}}}

\newcommand{\bZ}{{\mathbb{Z}}}

\def\1{\boldsymbol {|}}
\newcommand{\Def}{\mathrel{\mathop:}=}

\newcommand{\Hess}{\operatorname{Hess}}

\newcommand{\mes}{\operatorname{mes}}

\newcommand{\Tr}{\operatorname{Tr}}

\newenvironment{claim}[1][{\textup{(\theequation)}}]{\refstepcounter{equation}\vglue10pt
\begin{trivlist}
\item[{\hskip\labelsep#1}]}{\vglue10pt\end{trivlist}}

\newenvironment{claim*}[1][{}]{\vglue10pt
\begin{trivlist}
\item[{\hskip\labelsep#1}]}{\vglue10pt\end{trivlist}}

\newenvironment{phantomequation}[1][]{\refstepcounter{equation}}{}
\newcounter{note}

\DeclareTextCommand{\textinfty}{PU}{\9042\036}

\DeclareTextCommand{\textge}{PU}{\9042\145}
\DeclareTextCommand{\textle}{PU}{\9042\144}
\DeclareTextCommand{\texthat}{PD1}{\136}

\begin{document}
\title{Schr\"{o}dinger Operator with Strong Magnetic Field: Propagation of singularities and sharper asymptotics}
\author{Victor Ivrii}

\maketitle
{\abstract%
We consider $2$-dimensional Schr\"odinger operator with the non-degenerating magnetic field and we discuss spectral asymptotics with the remainder estimate  $o(\mu^{-1}h^{-1})$ or better.

We also consider $3$-dimensional Schr\"odinger operator with the non-degenerating magnetic field and we discuss spectral asymptotics with the remainder estimate  $o(\mu^{-1}h^{-1})$ or better.

We also consider generalized Schr\"odinger-Pauli operator in the same framework albeit with $\mu h\ge 1$ and derive  spectral asymptotics with the remainder estimate up to $O(h^{-1}|\log h|)$ and with the principal part $\asymp \mu h^{-2}$.
\endabstract}

\setcounter{chapter}{-1}
\chapter{Introduction}
\label{sect-13-0}

Our goal is to derive spectral asymptotics of $2,3$-dimensional Schr\"odinger operator 
\begin{equation}
A=\sum_{j,k} P_jg^{jk}P_k+V,\qquad \text{with\ \ } P_j=hD_j-\mu V_j
\label{13-1-1}
\end{equation}
near the boundary
where
\begin{claim}\label{13-1-2}
$g^{jk}=g^{kj}$, $V_j$ and $V$ are real-valued functions, $ h\in (0,1]$, $\mu \in [1,\infty )$.
\end{claim}
and in $B(0,1)\subset X$ the following conditions are fulfilled:
\begin{phantomequation}\label{0-3}\end{phantomequation}
\begin{equation}
|D^\alpha g^{jk}|\le c,\quad |D^\alpha V_j|\le c,\quad |D^\alpha V|\le c
\quad \qquad \forall \alpha :|\alpha |\le K,
\tag*{$\textup{(\ref*{0-3})}_{1-3}$}
\end{equation}
\begin{equation}
\epsilon _0\le \sum_{j,k} g^{jk}\eta _j\eta _k\cdot |\eta |^{-2}\le c
\quad \forall \eta \in \bR^d\setminus 0 \quad \forall x \in B(0,1)
\label{13-1-4}
\end{equation}
and unless opposite is specified, we always assume that
\begin{equation}
X \supset B(0,1).
\label{13-1-5}
\end{equation}

So we basically want  to generalize results of sections~\ref{book_new-sect-13-3}--\ref{book_new-sect-13-5} of
Chapter~\ref{book_new-sect-13}~\cite{futurebook}\footnote{\label{foot-13-0} This article is a rather small part of the huge project to write a book and is just part of the section~\ref{book_new-sect-13-6} consisting entirely of newly researched results. Chapter~\ref{book_new-sect-13} corresponds to Chapter~6 of its predecessor V.~Ivrii~\cite{Ivr1}. External references by default are to  \cite{futurebook}.}. We assume that condition \begin{equation}
F\ge \epsilon _0
\label{13-2-1}
\end{equation}
where $F$ is the scalar intensity of magnetic field defined by $\textup{(\ref{book_new-13-1-13})}_{2,3}$,  as $d=2,3$ respectively.

What we want is to analyze how improved dynamics results in the sharper remainder estimates.

\section*{Plan of the article}

We start with the two-dimensional case when singularities propagate along magnetic drift lines and then proceed to the three-dimensional case when singularities propagate along magnetic lines which could however have ``side-drift''.

In the main section~\ref{sect-13-6-3} we consider supersharp asymptotics as $d=3$ with the remainder estimate $O(\mu^{-1}h^{-2})$ as $\mu h\le 1$ and $O(h^{-1})$ as $\mu h\ge 1$.

\chapter{Case $d=2$}
\label{sect-13-6-1}

\section{Classical dynamics and heuristics}
\label{sect-13-6-1-1}

Let us start from the classical dynamics. We assume that potential is truly generic i.e. no special conditions with $\nu$ are assumed. To make things simpler we assume at the moment that $g^{jk}=\updelta_{jk}$ and $F=\const$.

Looking at the canonical form we conclude that there is fast circular movement in $(\mu x_2,\xi_2)$ and there is a slow drift movement in $(x_1,\mu^{-1}\xi_1)$. More precisely, consider first two terms of the canonical form:
\begin{equation}
\omega_0 (x_1,\mu^{-1}\xi_1)+ \omega_1(x_1,\mu^{-1}\xi_1) (\mu^2 x_2^2+\xi_2) + O(\mu^{-2})
\label{13-6-1}
\end{equation}
and we understand that along $(\mu x_2,\xi_2)$ there is a circular movement with period $T(\mu)=\pi \mu^{-1}\omega_1^{-1}+O(\mu^{-3})$ and along $(x_1,\mu^{-1}\xi_1$ there is a drift with the velocity
$\mu^{-1}H_{\omega_0 + \omega_1 k}$ with Hamiltonian calculated with respect to $(x_1,\xi_1)$ only and with $k=(\tau-\omega_0) \omega_1$ plugged after (as we are interested at energy level $0$).

Going back with symplectomorphism $\bar{\Psi}$ we conclude that modulo $O(\mu^{-3})$
\begin{equation}
T(\mu)=\pi \mu^{-1}F^{-1}
\label{13-6-2}
\end{equation}
and the drift speed is
\begin{equation}
\mathbf{v}_{\mathsf{drift}}= \mu^{-1}\bigl(\nabla \bigl(\frac{V-\tau}{F}\bigr)\bigr)^\perp
\label{13-6-3}
\end{equation}
which defines \emph{magnetic drift lines\/}\index{magnetic drift lines\/} which are level lines of
$(V-\tau)F^{-1}$~\footnote{\label{foot-13-34} In the special case $F=\const$ magnetic drift lines are level lines of $V$.}.

\begin{remark}\label{rem-13-6-1}
Note that in an invariant form
\begin{equation}
F= \frac{1}{\sqrt{g}} |F_{12}|,\quad F_{12}=\partial_1 V_2-\partial_2 V_1,\quad
g= (g^{11}g^{22}-g^{12}g^{21})^{-1}
\label{13-6-4}
\end{equation}
and magnetic drift is described by
\begin{equation}
\frac{dx_1}{dt}= \frac{1}{\sqrt g} \partial_2 \frac{V}{f},\qquad
\frac{dx_2}{dt}= -\frac{1}{\sqrt g} \frac{V}{f}.
\label{13-6-5}
\end{equation}
where
\begin{equation}
f=\frac{1}{\sqrt{g}} F_{12}
\label{13-6-6}
\end{equation}
is a \emph{pseudo-scalar\/}\footnote{\label{foot-13-35} I.e. it changes sign as we change orientation of the coordinate system.}.
\end{remark}

\begin{remark}\label{rem-13-6-2}
Magnetic drift lines are terminated in the critical points of $(V-\tau)F^{-1}$~\footnote{\label{foot-13-36} In the special case $F=\const$ those are critical points of $V$.}.
\end{remark}

If we are interested in the time $T\gg \mu$ we need to follow magnetic drift line either close to termination point, or to large distance, or assume that it is closed smooth line. However we need to remember that Hamiltonian trajectories follow drift lines albeit do not coincide with them!

\begin{example}\label{ex-13-6-3}
Let $g^{jk}=\updelta_{jk}$, $F=1$.

\medskip\noindent
(i) As $V=x_2$ magnetic drift lines are just straight lines $x_2=\const$ with no termination points and Hamiltonian trajectories are cycloids of figure~\ref{2D-model}(b).

\begin{figure}[h]
\centering
\subfloat[Periodic]{\includegraphics[scale=1]{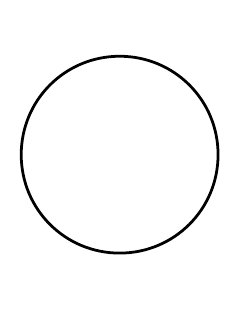}} \qquad
\subfloat[Drift]{\includegraphics[scale=1]{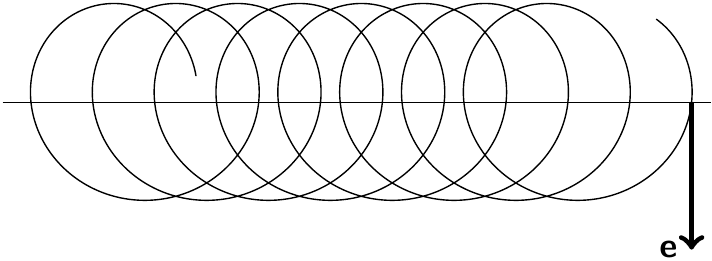}}
\caption{\label{2D-model} Trajectories for $2$-dimensional model operator:\newline a) Unperturbed b) Perturbed by a constant electric field}
\end{figure}

\vglue-10pt
\begin{figure}[h]
\centering
\subfloat[Helix]{\includegraphics[scale=.5]{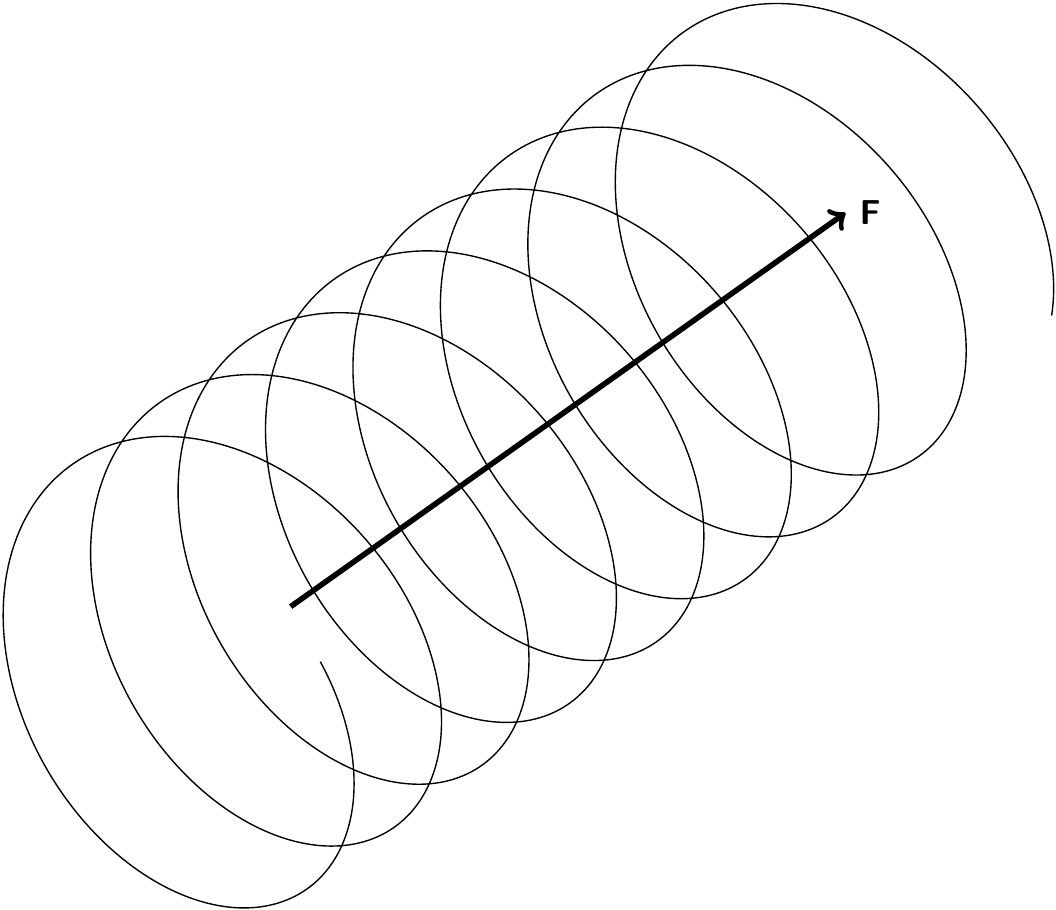}} \qquad
\subfloat[Deformed helix]{\includegraphics[scale=.5]{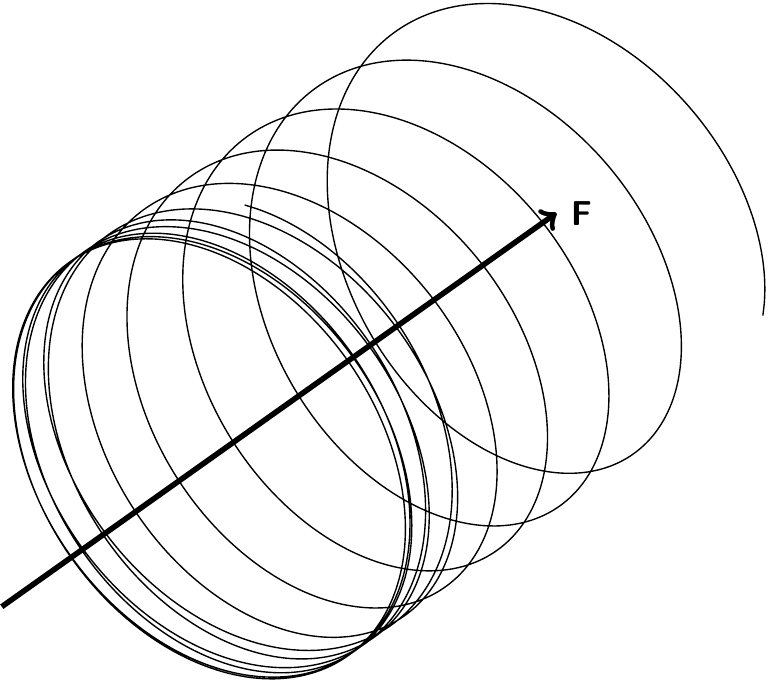}}
\caption{\label{3D-model} Trajectories for $3$-dimensional model operator:\newline a) Unperturbed b) Perturbed by a constant electric field parallel to magnetic lines: non-uniform movement along magnetic lines; drift is negligible in comparison}
\end{figure}

\medskip\noindent
(ii) Consider $V=\pm ( x_1^2 + x_2^2)$. Then magnetic drift lines are circles
$x_1^2+x_2^2=\const$. On the other hand, Hamiltonian trajectories are cycloids with frequencies $2\mu (1+O(\mu^{-2})$ and $2\mu^{-1}+(1+O(\mu^{-2})$ and they are periodic only as $\mu=\mu_k = k^{\frac{1}{2}} (1+ O(k^{-1})$ with 
$k\in \bQ^+$. See figure~\ref{ex-2D}(a),(b).

\medskip\noindent
(iii) Consider $V=x_1x_2$. Then magnetic drift lines are hyperbolas
$x_1x_2=\const$. On the other hand, Hamiltonian trajectories are shown on figure~\ref{ex-2D}(c),(d), one which terminates in $x=0$ and another which bypasses it.
\begin{figure}[h]
\centering
\subfloat[Periodic]{\includegraphics[width=0.45\linewidth]{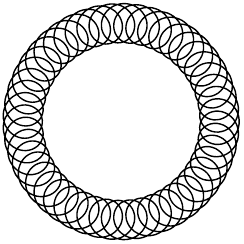}} \qquad
\subfloat[Non-periodic]{\includegraphics[width=0.45\linewidth]{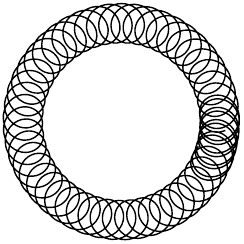}}

\subfloat[Termination]{\includegraphics[width=0.55\linewidth]{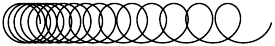}} \qquad
\subfloat[Non-termination]{\includegraphics[width=0.35\linewidth]{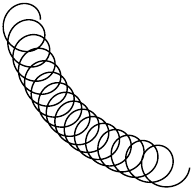}}
\caption{\label{ex-2D} Trajectories for $2$-dimensional model operator to example~\ref{ex-13-6-3}(ii) (in (a),(b)) and~\ref{ex-13-6-3}(iii) (in (c),(d))}
\end{figure}
\end{example}

Our conjecture is that we can derive remainder estimate $o\bigl(\mu^{-1}h^{-1}+1\bigr)$ in all the cases when drift lines are not closed; if drift lines are closed we can derive this remainder estimate for $\mu$ except belonging to rather thin exceptional set.

One should notice that magnetic drift lines escaping to infinity are rather dangerous as then $V$ does not grow which may indicate presence of the essential spectrum.

\chapter{Case $d=3$}
\label{sect-13-6-2}

Now we consider $3$-dimensional case.

\section{Preliminaries and improved reduction}
\label{sect-13-6-2-1}
It is well known from physics that the classical three-dimensional particles in a strong magnetic field move almost along magnetic lines.%
\index{magnetic!line}%
\index{line!magnetic}
So it is not surprising that in our analysis magnetic lines are also important. Let us recall that a magnetic line of the magnetic field $\mathbf{F}=(F^1,F^2,F^3)$ is an integral curve $x=x(t)$ of the field
$F^{-1}\mathbf{F}$. Thus, $x(t)$ is given by the system
\begin{equation}
{\frac{dx_j}{dt\ }}={\frac{2}{F}}F^j\qquad (j=1,2,3)
\label{13-6-32}
\end{equation}
where
\begin{equation}
F=\bigl( \sum_{j,k} g_{jk}F^jF^k\bigr)^{\frac{1}{2}}=
\frac{1}{2} \bigl(\sum_{j,k,l,m} g^{jk}g^{lm}F_{jl}F_{km}\bigr)^{\frac{1}{2}}
\label{13-6-33}
\end{equation}
is the scalar intensity; let us introduce the corresponding flow $\Phi _t$ on $X$.

It follows from the definition of $\mathbf{F}$ that this field is solenoidal, i.e.,
\begin{equation}
\sum_{j}\partial _j(F^j\sqrt{g})=0
\label{13-6-34}
\end{equation}
and therefore
\begin{claim}\label{13-6-35}%
\emph{Magnetic flow\/}\index{flow!magnetic}\index{magnetic!flow} $\Phi_t$ preserves the density $\sqrt{g}dx$ on $X$.
\end{claim}

Moreover,
\begin{claim}\label{13-6-36}%
Magnetic flow $\phi_t$ can be lifted to the Hamiltonian flow $\tilde{\Phi}_t$ on $\Lambda $ which is given by the Hamiltonian
\begin{equation}
{\frac{1}{F}}\sum_j F^j\bigl(\xi _j-\mu V_j(x)\bigr).
\label{13-6-37}%
\end{equation}
\end{claim}

\begin{remark}\label{rem-13-6-12}
We could define the flow $\Phi _t$ by the system
\begin{equation}
\frac{dx_j}{dt} = \frac{2}{F} \omega F^j \quad (j=1,2,3)
\tag*{$\textup{(\ref*{13-6-32})}'$}\label{13-6-32-'}
\end{equation}
with functions $\omega$ and $\omega^{-1}$ regular along the trajectory; then $\tilde{\Phi}_t$ would be given by the Hamiltonian
\begin{equation}
\frac{\omega}{F} \sum_j F^j\bigl(\xi _j-\mu V_j(x)\bigr)
\tag*{$\textup{(\ref*{13-6-37})}'$}\label{13-6-37-'}
\end{equation}
and the invariant density would be $\omega^{-1} F \sqrt{g} dx$.
\end{remark}

\begin{remark}\label{rem-13-6-14}
Even from the heuristic point of view there is an essential difference:

\medskip\noindent
(i) Magnetic lines depend only on ${\mathbf F}$ while magnetic drift lines depend on $V$ as well. In order to the get magnetic drift lines at the energy level $\tau$ one should replace $V$ by $V-\tau $. Then, for $F\ne \const$ the magnetic drift lines depend on the energy level as well.

\medskip\noindent
(ii) As soon as magnetic lines are defined, the movement along them is defined by the quantum number $j$ and the energy level $\tau $. Moreover, the direction of the movement depends on (\ref{13-6-37}). In contrast, in the two-dimensional case the choice of $j$ is automatic and the movement is defined only by the energy level which should be fixed. System (\ref{13-6-5}) is a dynamical system while (\ref{13-6-32}) only describes geometrical shape.

\medskip\noindent
(iii) The movement along magnetic drift lines is given in scaled time; the real
time is $\mu t$. According to the theory of sections~\ref{book_new-sect-2-2} and~~\ref{book_new-sect-2-4}, $t$ should not be too large. Moreover, under condition \begin{equation}
|D^\alpha \bigl({\frac{V}{F}}\bigr)|\le c\nu \qquad \forall \alpha :1\le |\alpha |\le K,
\label{13-3-85}
\end{equation}
it is sufficient to assume that $\nu t$ is not too large (because we can divide operator by $\nu $ and then the scaled time will be $\nu t$).

\medskip\noindent
(iv) $Vf^{-1}$ is constant along magnetic drift lines and for big enough $\mu $
(e.g., $\mu \ge h ^{-\frac{1}{2}}$) one need only consider lines with
$-V/(\mu hF)\in 2\bZ ^++1$.
\end{remark}

\chapter{Supersharp estimates as $d=3$}
\label{sect-13-6-3}

\section{Framework}
\label{sect-13-6-3-1}

Now instead of escape conditions we assume that all trajectories of $\Psi_{r,t}$ are trapped \underline{either} because all magnetic lines are periodic \underline{or} by the growing potential. We discuss only the latter case but the former could be analyzed in the same way. Idea is that magnetic lines drift with a speed $O(\mu^{-1})$ and because of this drift we can increase $T^*$ from $\asymp 1$ to $\asymp \mu$.

So, let us consider coordinate system in which $F^1=F^2=0$, $F^3\ne 0$ and
\begin{equation}
\frac{V}{F} > 0\qquad\text{as\ \ } |x'| <1, \ |x_3|\ge C_0
\label{13-6-89}
\end{equation}
where $x'=(x_1,x_2)$.

Then $V/F$ reaches its minimum $z^0(x')$ with respect to $x_3$ at $[-C_0,C_0]$; we assume that this minimum is unique and non-degenerate we also assume that it is unique and moreover
\begin{equation}
\epsilon_0 \le \bigl(x_3- z^0(x')\bigr)^{-1}\partial_{x_3}\frac{V}{F} \le C_0.
\label{13-6-90}
\end{equation}
We also assume first that $|\nabla (V/F)|\ge \epsilon$. Then we can extend time $T_*=h^{1-\delta}$ to $T^*=\epsilon$ but now we will try to go beyond using canonical form. First we have a fast magnetron movement in $(x_3,\xi_3)$, then an oscillatory movement in $(x_2,\xi_2)$ and finally a slow magnetic drift movement in $(x_1,\xi_1)$.

Let us denote by $T(x,r)$ the oscillation time with respect to $(x_2,\xi_2)$ on energy level $0$. Then the drift velocity is $\mu^{-1}H'_{b_0+b_1r}$ with $b_0=V\circ \bar{\Psi}$, $b_1=F\circ \bar{\Psi}$ where $H$ means that we consider Hamiltonian field only with respect to $(x_1,\xi_1)$.

Therefore $(x_1,\xi_1)$-shift will be
\begin{equation}
\mu^{-1}\int_0^{T(x_1,\xi_1,r)} H'_{b_0+b_1r}\,dt =
\mu^{-1}\int_{z^-}^{z^+} \frac{H'_{b_0+b_1r}}{\sqrt{-b_0-b_1r}}\,dx_2
\label{13-6-91}
\end{equation}
which in turn is equal to
$-2\mu^{-1}H'_\eta +O(\mu^{-2})$ with
\begin{equation}
\eta (x_1,\xi_1,r)=
\int _{z^-}^{z^+} \bigl( -b_0 - b_1 r\bigr)^{\frac{1}{2}}\, dx_2
\label{13-6-92}
\end{equation}
where $z^\pm=z^\pm (x_1,\xi_1,r)$ are roots of
\begin{equation}
b_0(x_1,\xi_1,z)+ b_1(x_1,\xi_1,z)r=0,
\label{13-6-93}
\end{equation}
$z^- <z^0< z^+$, and 
\begin{claim}\label{13-6-94}
We use the same notation for functions of $(x_1,\xi_1)$ and for functions of $x'=(x_1,x_2)$ as $(x_1,x_2)=\bar{\Psi}(x_1,\xi_1)$,
\end{claim}
\begin{gather}
r\le \bar{r}(x_1,\xi_1)= -\frac{b_0}{b_1}(x_1,\xi_1,z^0(x_1,\xi_1))
\label{13-6-95}
\\
\shortintertext{and}
T(x_1,\xi_1,r)= \int_{z^-}^{z^+} \frac{dx_2}{\sqrt{-b_0-b_1r}}
\label{13-6-96}
\end{gather}
is a period of oscillations.
Note that due to (\ref{13-6-90})
\begin{claim}\label{13-6-97}
$|H_\eta|\asymp 1$ as $z^\pm$ are close to $z^0(x_1,\xi_1)$
\end{claim}
and therefore shift is observable: $\mu^{-1}\gg \mu^{-1}h$. The same is true for any $r$ provided $|\nabla \eta|\asymp 1$ for all $r\le \bar{r}(x_1,\xi_1)$.

Note that we can rewrite this condition
\begin{gather}
|\nabla_{x'} \eta (x',r)|\ge \epsilon_0 \qquad
\forall x', r\le \bar{r}(x')
\label{13-6-98}
\\
\shortintertext{with}
\eta (x',r)= \int _{\ell (x')} \bigl( -V - F r\bigr)^{\frac{1}{2}}\, ds
\label{13-6-99}
\end{gather}
where $x'$ coordinates on the surface transversal to magnetic lines and $ds$ is an element of length along such lines where $\ell(x',r)$ indicates the segment of the magnetic line $\Phi_t(x')$ where $V+rF<0$. Our ``new'' $\eta$ is connected with ``old'' $\eta$ in the obvious way.

Therefore $T^*\asymp 1$ is upgraded to $T^*\asymp \mu$. 

\section{Case $\mu h\le 1$}
\label{sect-13-6-3-2}

The fact that the standard under condition (\ref{13-6-101}) below $T^*\asymp 1$ can be replaced by $T^*\asymp \mu$ without adding new singularities implies that the standard remainder estimate $\R^\T\le Ch^{-2}$ and 
$\R^\MW_\infty\le Ch^{-2}$  could be upgraded to $\R^\T\le C\mu^{-1}h^{-2}$ and $\R^\MW_\infty\le C\mu^{-1}h^{-2}$ respectively as $\mu h\le 1$.  So we arrive to

\begin{theorem}\label{thm-13-6-30}
Let conditions \textup{(\ref{13-1-1})}--\textup{(\ref{13-1-5})}, \textup{(\ref{13-2-1})}, 
\begin{gather}
|V|+|\nabla V|\ge \epsilon _0\qquad \forall x\in B(0,1);
\label{13-3-46}
\shortintertext{and} 
|\nabla \frac{V}{F} |\ge \epsilon_0
\label{13-3-54}
\end{gather}
be fulfilled.

Further, let conditions \textup{(\ref{13-6-89})}, \textup{(\ref{13-6-90})} and \textup{(\ref{13-6-98})} be fulfilled. Finally let non-degeneracy condition
\begin{equation}
|\nabla \bigl(\frac{V}{F}\bigr)|\le \epsilon \implies
|\det \Hess \bigl(\frac{V}{F}\bigr)|\ge \epsilon
\label{13-6-100}
\end{equation}
be fulfilled. Then as $\mu h\le 1$ estimate
\begin{equation}
\R^\MW \le C\mu^{-1}h^{-2}
\label{13-6-101}
\end{equation}
holds.
\end{theorem}

\begin{proof}
We leave to the reader an easy exercise to prove that all extra terms in $\R^\MW_\infty$ could be dropped so $\R^\MW_\infty$ becomes $\R^\MW$.
\end{proof}

This theorem is  illustrated by example~\ref{ex-13-6-34}(i). On the other hand, assume that $\eta$ has a critical point. Let us introduce  $\gamma$-admissible partition of unity with respect to $(x_1,\xi_1)$ with 
\begin{equation}
\gamma= \gamma (x_1,\xi_1, r)\Def 
\max\bigl(\epsilon |\nabla \eta | ,C\bar{\gamma}\bigr).
\label{13-6-102}
\end{equation}
Then the shift during period is of magnitude $\mu^{-1}\gamma$ and uncertainty principle requires $\mu^{-1}\gamma \times \gamma \gg \mu^{-1}h$ or 
$\gamma \gg h^{\frac{1}{2}}$. Therefore selecting $\bar{\gamma}=h^{\frac{1}{2}-\delta}$ we conclude that the contribution of the zone $\{\gamma \ge \bar{\gamma}\}$ to $\R^\MW $ does not exceed $C\mu^{-1}h^{-2}$. Meanwhile contribution of the zone 
$\{\gamma \le \bar{\gamma}\}$ does not exceed 
$Ch^{-2}\mes (\{x:\ |\nabla \eta (x)|\le \bar{\gamma}\})$ which is 
$Ch^{-2}\bar{\gamma}^2$ under assumption (\ref{13-6-107}) below. Then we arrive to  the remainder estimate 
\begin{equation}
\R^\MW \le C\mu^{-1}h^{-2}+ Ch^{\delta -1}
\label{13-6-103}
\end{equation}
with an arbitrarily small exponent $\delta>0$.

One can improve it by using logarithmic uncertainty principle; then we can select $\bar{\gamma}=C(h|\log h|)^{\frac{1}{2}}$ leading us to the remainder estimate
\begin{equation}
\R^\MW \le C\mu^{-1}h^{-2}+ Ch^{-1}|\log h|
\label{13-6-104}
\end{equation}

However we can do better than this. Let us select 
$\bar{\gamma}= Ch^{\frac{1}{2}}$. The following problem is easily accessible by our methods:

\begin{problem}\label{problem-13-6-31}
(i) Prove that as $Q'=Q'(x_1,\mu^{-1}hD_1,j)$, $Q''=Q''(x_1,\mu^{-1}hD_1,j)$  are $\gamma$-admissible elements with $\gamma\ge C_0\bar{\gamma}$ then
\begin{equation}
|F_{t\to h^{-1}\tau}\chi_T(t)\Gamma Q'_{x'}\sU_{j,j}{\,^t\!}Q''_y|
\le C \mu h^{-1} \gamma^2 T \bigl(\frac{\bar{\gamma}}{T \gamma}\bigr)^s\Bigr)
\label{13-6-105}
\end{equation}

\medskip\noindent
(ii) Then prove that the contribution of $\gamma$-element to $\R^\T$ with $T=\epsilon$ does not exceed
\begin{equation}
Ch^{-2}\gamma^2 \Bigl( \mu^{-1} + \bigl(\frac{\bar{\gamma}}{\gamma}\bigr)^s\Bigr);
\label{13-6-106}
\end{equation}
\end{problem}

Then we arrive to the following 

\begin{theorem}\label{thm-13-6-32}
Let all conditions of theorem~\ref{thm-13-6-30} except \textup{(\ref{13-6-98})} be fulfilled. Then under condition
\begin{equation}
|\nabla_{x'} \eta (x',r)|\le \epsilon \implies |\det \Hess_{x'} \eta (x',r) |\ge \epsilon \qquad
\forall x', r\le \bar{r}(x')
\label{13-6-107}
\end{equation}
as $\mu h\le 1$ estimate \textup{(\ref{13-6-101})} holds.
\end{theorem}

\begin{remark}\label{remark-13-6-33}
(i) As $\epsilon_0 h^{-1}\le \mu \le h^{-1}$ we need to check conditions (\ref{13-6-98}), (\ref{13-6-107}) only $r= (2j+1)\mu h $; also condition (\ref{13-6-100}) should be checked only as $|V- (2j+1)\mu h F|\le \epsilon$;

\medskip\noindent
(ii) Estimate $o(h^{-2})$ holds for  $\mu h \le 1$ under weak non-degeneracy condition
\begin{equation}
\mes (\{x': \nabla \eta (x', r) =0\})=0\qquad \text{as\ \ }r=(2j+1)\mu h\ \ \forall j;
\label{13-6-108}
\end{equation}
(iii) Estimate $O\bigl(\mu^{-1}h^{-2}+h^{-\frac{3}{2}}\bigr)$ holds for  
$\mu h \le 1$ under intermediate non-degeneracy condition
\begin{multline}
|\nabla_{x'} \eta (x', r)  |\le \epsilon \implies | \Hess_{x'} \eta (x',r) |\ge \epsilon \\
\forall x', r=(2j+1)\mu h, j\in \bZ^+;
\label{13-6-109}
\end{multline}
\end{remark}

\begin{example}\label{ex-13-6-34}
Consider
\begin{equation}
a= (\xi_1-\frac{1}{2}\mu x_2)^2 + (\xi_2+\frac{1}{2}\mu x_1)^2 +V(x).
\label{13-6-110}
\end{equation}
(i) As
\begin{equation}
V= k^2 x_3^2 + l x_1 -\tau \qquad (\tau>0)
\label{13-6-111}
\end{equation}
$k>0$ ensures trapping and $l\ne 0$ ensures condition (\ref{13-6-98}).

\medskip\noindent
(ii) As
\begin{equation}
V= k^2 x_3^2 + l_1 x_1^2+l_2 x_2^2-\tau \qquad (\tau>0)
\label{13-6-112}
\end{equation}
$k>0$ ensures trapping and $l_1\ne 0$, $l_2\ne 0$ ensures condition (\ref{13-6-107}).

\medskip\noindent
(iii) In (ii) assume that $l_1=l_2=l^2>0$; then we have at least two oscillatory movements (with respect to $(x_2,\xi_2)$ and slow oscillatory drift movement with respect to $(x_1,\xi_1)$). Then we can increase $T^*$ to $T^*\gg \mu$ and under appropriate non-commensurability condition we can get remainder estimate $o(\mu^{-1}h^{-2})$ instead of $O(\mu^{-1}h^{-2})$.

As $\mu h\to 0$ the third periodic movement (with respect to $(x_3,\xi_3)$ in the canonical form) plays role and non-commensurability means that at least one of two commensurability conditions should be violated. This is fulfilled automatically for non-commensurable $k$ and $l$ while for commensurable $k$ and $l$ we get one non-commensurability condition to $\mu$. Then we get remainder estimate $o(\mu^{-1}h^{-2})$ instead of $O(\mu^{-1}h^{-2})$.

\medskip\noindent
(iv) Let
\begin{equation}
V=k^2 x_3^2-\tau \qquad (\tau>0).
\label{13-6-113}
\end{equation}
Then there is no drift but while fast rotations have frequency $2\mu$, normal oscillation have frequency $2k$ and unless they are commensurable we get non-periodic movement and as $\mu h\to 0$ we get remainder estimate $o(h^{-2})$ under appropriate assumption to $\mu$ as $\mu h \to +0$.
\end{example}

\begin{Problem}\label{problem-13-6-35}
(i) Investigate case when potential traps only trajectories with $r\ge \bar{r}$ and other trajectories escape. One should expect remainder estimate 
$\R^\MW= o(h^{-2})$;

\medskip\noindent
(ii) Investigate the case when all magnetic lines are periodic and $V/F$ is constant along them.

\medskip\noindent
(iii) Investigate the case when all magnetic lines are periodic and $V/F$ is not constant along them (then trajectories of $\Psi_{r,t}$ with $r\ge \bar{r}$ are trapped and other trajectories are periodic).
\end{Problem}

\begin{figure}[h]
\centering
\subfloat[Oscillation and drift]{\includegraphics[width=0.4\linewidth]{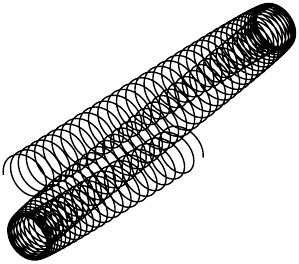}} \qquad
\subfloat[Periodic and drift]{\includegraphics[width=0.4\linewidth]{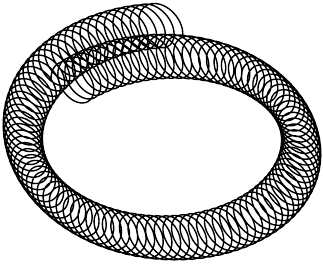}}
\caption{\label{weird} Trajectories for $3$-dimensional model operator:\newline a) Operator (\ref{13-6-110})--(\ref{13-6-111}) b) Operator (\ref{13-6-114})}
\end{figure}

\begin{example}\label{ex-13-6-36}
\begin{equation}
a= \xi_1^2+\xi_2^2 + (\xi_3 -\mu \alpha (|x'|))^2 +w(|x'|)-\tau
\label{13-6-114}
\end{equation}
Then $\mathbf{F}=(-x_2,x_1,0)f(|x'|)|x'|^{-1}$, $F=f(|x'|)$, $f(\rho) =\partial_\rho \alpha >0$ and and all magnetic lines are circles.

\medskip\noindent
(i) Then drift speed is $\partial_{\rho}\eta(\rho) $ with
$\eta = (-w(\rho) -r |f(\rho)|)$ and it is directed along $x_3$. Then under appropriate assumptions ensuring (\ref{13-6-98}) or (\ref{13-6-100}) one can derive corresponding asymptotics.

\medskip\noindent
(ii) As $w(\rho)=-\tau$ and $f(\rho)=\const$ there is no drift and again one can derive asymptotics $o(h^{-2})$ under non-commensurability condition as 
$\mu h\to 0$.
\end{example}

\section{Schr\"odinger-Pauli operator as $\mu h\ge 1$}
\label{sect-13-6-3-3}

Consider now Schr\"odinger-Pauli operator 
\begin{equation}
A=\sum_{j,k} P_jg^{jk}P_k+V-\mu h F,\qquad \text{with\ \ } P_j=hD_j-\mu V_j
\label{13-5-3}
\end{equation}
as $\mu h\ge 1$. We assume that modified condition (\ref{13-6-100})
\begin{equation}
\min_j |\frac{V}{F} -2j\mu h|+ |\nabla \bigl(\frac{V}{F}\bigr)|\le \epsilon \implies
|\det \Hess \bigl(\frac{V}{F}\bigr)|\ge \epsilon
\tag*{$\textup{(\ref*{13-6-100})}'$}\label{13-6-100-'}
\end{equation}
is fulfilled. Then exactly the same arguments as in the proofs of theorems \ref{thm-13-6-30}, \ref{thm-13-6-32} lead us to the following 

\begin{theorem}\label{thm-13-6-37}
Consider Schr\"odinger-Pauli operator \textup{(\ref{13-5-3})} with $\mu h\ge 1$. 
Let conditions \textup{(\ref{13-1-1})}--\textup{(\ref{13-1-5})}, \textup{(\ref{book_new-13-2-1})}, \textup{(\ref{13-3-46})} and \textup{(\ref{13-3-54})} be fulfilled.

Further, let conditions \textup{(\ref{13-6-89})},  \textup{(\ref{13-6-90})} and \ref{13-6-100-'} be fulfilled. Then

\medskip\noindent
(i) Under assumption
\begin{equation}
|\nabla_{x'} \eta (x',r)|\ge \epsilon_0 \qquad
\forall x', r\le \bar{r}(x')\qquad  \forall x', r=2j\mu h,\  j\in \bZ^+
\tag*{$\textup{(\ref*{13-6-98})}'$}\label{13-6-98-'}
\end{equation}
estimate 
\begin{equation}
\R^\MW_1 \le C h^{-1} +C'_s \mu h^s
\label{13-6-115}
\end{equation}
where $\R^\MW_1$ is due to the perturbation term $\asymp h^2$ in the canonical form\footnote{\label{foot-13-45} See remark~\ref{book_new-rem-13-5-13}(ii). We should not care about this before as the sharpest remainder estimate was $O(\mu h^{-1+\delta})$.};

\medskip\noindent
(ii) Estimate
\begin{equation}
\R^\MW  \le C \mu 
\label{13-6-116}
\end{equation}
holds under condition 
\begin{multline}
|\nabla_{x'} \eta (x',r)|\le \epsilon \implies |\det \Hess_{x'} \eta (x',r) |\ge \epsilon \\
\forall x', r=2j\mu h,\  j\in \bZ^+
\tag*{$\textup{(\ref*{13-6-107})}'$}\label{13-6-107-'}
\end{multline}
\end{theorem}

\begin{proof}
We leave details to the reader.
\end{proof}

\begin{remark}\label{rem-13-6-38}
(i) Estimate $\R^\MW =o(\mu h^{-1})$ holds  under weak non-degeneracy condition
\begin{equation}
\mes (\{x': \nabla \eta (x', r) =0\})=0\qquad \text{as\ \ }r=2j\mu h\ \ \forall j;
\tag*{$\textup{(\ref*{13-6-108})}'$}\label{13-6-108-'}
\end{equation}

\medskip\noindent
(ii) Estimate $\R^\MW=O(\mu h^{-\frac{1}{2}})$ holds under intermediate non-degeneracy condition
\begin{multline}
|\nabla_{x'} \eta (x', r)  |\le \epsilon \implies | \Hess_{x'} \eta (x',r) |\ge \epsilon \\
\forall x', r=2j\mu h, j\in \bZ^+;
\tag*{$\textup{(\ref*{13-6-109})}'$}\label{13-6-109-'}
\end{multline}
\end{remark}

\section{Schr\"odinger-Pauli operator as $\mu h\ge 1$ reloaded}
\label{sect-13-6-3-4}

While most likely one can improve (\ref{13-6-115})  to 
\begin{equation}
\R^\MW_1 \le C h^{-1} +C'  \mu \exp (\epsilon h^{-1})
\label{13-6-117}
\end{equation}
we still want to improve it and (\ref{13-6-116}) further. To do this we consider operator in the reduced form; as $\mu h\ge C$ we need to consider only operator
$\cA_0(x_1,\hslash D_1; x_2, hD_2)$ with $\hslash=\mu^{-1}h$ and we consider it as $\hslash$-pseuodo-differential operator with an operator-valued symbol $\mathbf{a}(x_1,\xi_1; x_2,hD_2)$ acting in the axillary space $\bH=\sL(\bR_{x_2})$. Then under condition (\ref{13-6-98}) it is microhyperbolic which leads us to statements (i)--(ii)  of the following

\begin{theorem}\label{thm-13-6-39}
(i) Let conditions of theorem~\ref{thm-13-6-37}(i) be fulfilled. Then as 
$h\in (\mu^{-1},h_0)$ where $h_0$ is small enough constant
\begin{multline}
\R^\MW_2\Def \\ |\int \Bigl(e(x,x,0)- (2\pi )^{-1}\mu h^{-1} \mathsf{e}(x_1,x_2;x_3,x_3;0)\Bigr)\psi(x)\,dx|
\le Ch^{-1}
\label{13-6-118}
\end{multline}
where $\mathbf{e}(x';x_3,y_3;\tau)$ is the Schwartz kernel of the spectral projector  $\mathbf{e}(x';\tau)$ of operator 
$\mathbf{a}(x')=\mathbf{a}(x_1,x_2; x_3,hD_3)$ and as usual we identify objects depending on $x$ and on $(x_1,\xi_1,x_2)$ through $\bar{\Psi}$;

\medskip\noindent
(ii) In particular, for $\psi=\psi(x')$ 
\begin{equation}
\R^\MW_2= |\int \Bigl(\int e(x,x,0)\,dx_3 - (2\pi )^{-1}\mu h^{-1} \mathsf{n}(x_1,x_2;0)\Bigr)\psi(x')\,dx'|
\label{13-6-119}
\end{equation}
where $\mathbf{n}(x';\tau)=\Tr_{x_3}\mathbf{e}(x';\tau) $ is an eigenvalue counting function of operator $\mathbf{a}(x')$.
\end{theorem}

\begin{proof}
Proof is standard: as the propagation speed with respect to $(x_1,\xi_1)$ is $\asymp \mu^{-1}$ we can take $T^*\asymp \mu$ and $T_*\asymp \mu^\delta h$ with an arbitrarily small exponent $\delta>0$ and then we can launch successive approximations with unperturbed operator frozen as $x_1=y_1$; then $(x_1-y_1)=O(\mu^{-1}T_*)$.

We leave details to the reader. Note only that we need $h\in (\mu^{-1},h_0)$ because actually $\eta =\eta (x_1,\xi_1,h)$ and to have $|\nabla \eta |\asymp 1$ we need the above assumption.
\end{proof}

Assume now that condition (\ref{13-6-107}) is fulfilled instead. 

\begin{theorem}\label{thm-13-6-40}
Let conditions of theorem~\ref{thm-13-6-37}(ii) be fulfilled. Let 
$h\in (\mu^{-1},h_0)$. Then  estimate 
\begin{equation}
\R^\MW_2 \le Ch^{-1} |\log \mu|
\label{13-6M-117}
\end{equation}
with an arbitrarily small exponent holds.
\end{theorem}

\begin{proof}
Proof is standard. Consider $\eta=\eta (x_1,\xi_1,h)$ which as  
$h\in (\mu^{-1},h_0)$ has non-degenerate critical point.

Then propagation speed with respect to $(x_1,\xi_1)$ speed is 
$\asymp \mu^{-1}\gamma $ we can take $T^*\asymp \mu  $ and 
$T_*\asymp |\log \mu|  h/\gamma^2$  and then we can launch successive approximations with unperturbed operator frozen as $x_1=y_1$; then $(x_1-y_1)=O(\mu^{-1}\gamma T_*)$. Here we must assume that $\gamma \ge \bar{\gamma}=\mu^{-\frac{1}{2}}|\log \mu|$ and contribution to the remainder of the zone $\{\gamma \le \bar{\gamma}\}$ does not exceed 
$C\mu h^{-1}\bar{\gamma}^2$.

We leave details to the reader. 
\end{proof}

\begin{example}\label{ex-13-6-41}
(i) Example~\ref{ex-13-6-34}(i),(ii) illustrates theorems \ref{thm-13-6-39}, \ref{thm-13-6-40}.

\medskip\noindent
(ii) In particular, consider example~\ref{ex-13-6-34}(ii). Then (as $\psi=1$) the number of eigenvalues below $0$ equals to the number of the lattice points 
$\{ (2i+1) l \mu^{-1}h ,(2j+1)k h\}$ in the triangle $\{z_1\ge 0, z_2\ge 0, z+1+z_2<\tau\}$. Transition to ``Weyl'' expression  with respect to $(x_1,\xi_1)$ means that we replace summation with respect to $i$ by integration thus making an error $O(1)$ for each $j$ and the total error $O(h^{-1})$, resulting in the expression
$\const \cdot \mu h^{-1} \sum_j \bigl(\tau -(2j+1)h\bigr)_+$.
\end{example}

\section{Introducing boundary}
\label{sect-13-7-3-2}

According to section~\ref{book_new-sect-13-7} of \cite{futurebook} as boundary does not break remainder estimate $O(h^{-2}+\mu h^{-1})$ as $d=3$. Let us discuss sharper remainder estimates.

Magnetic billiards must follow magnetic lines; but magnetic lines are not billiards and all of them bounce back from the boundary going the same path in the opposite direction. So escape condition in the bounded domain fails for sure but these magnetic lines drift. 

Consider what happens as Hamiltonian trajectory reflects from the boundary. We need to analyze only the following
\begin{example}\label{ex-13-7-12}
Let 
\begin{equation}
a(x,\xi)= \frac{1}{2}\bigl(x_1^2 + (\xi_2-\mu x_1)^2+ \xi_3^2\bigr)
\label{13-7-40}
\end{equation}
and $X=\{x: x_3>k x_1\}$ i.e. we assume that magnetic lines are transversal to $\partial X$.

Then before reflection movement is described by 
\begin{multline}
x_1= \mu^{-1}\rho \sin (\mu t)+\bar{x}_1, \quad x_2=\mu^{-1}\rho \cos (\mu t) +\bar{x}_2, \quad x_3=\sigma t+\bar{x}_3,\\
\xi_1= \rho \cos (\mu t), \quad \xi_2= \mu \bar{x}_1, \quad \xi_3=\sigma
\label{13-7-41}
\end{multline}
with constant $\bar{x}$, $\rho\ge 0$, $\sigma$ and after reflection movement is described by $\textup{(\ref{13-7-41})}'$ with parameters $\bar{x}$, $\rho'$, $\sigma'$ and with $t$ replaced by $t'$.

One can prove easily that 
\begin{gather}
\rho^2+\sigma^2 = \rho^{\prime\,2}+\sigma^{\prime\,2}\label{13-7-42}\\
\shortintertext{and}
\bar{x}_1'=\bar{x}_1,\quad,\bar{x}_2'=\bar{x}_2-2\mu^{-1}l k, \quad \sigma'=\sigma -2l\label{13-7-43}
\end{gather} 
with  $l\in [\sigma-\rho|k|,\sigma+\rho |k|]$ the cosine of the incidence angle.

So, if $k=0$ i.e. $\mathbf{F}$ is orthogonal to $\partial X$ nothing really happens: $\bar{x}'=\bar{x}$, $\rho'=\rho$ and $\sigma'=-\sigma$.

However if $k\ne 0$ i.e. $\mathbf{F}$ is not orthogonal to $\partial X$, $\bar{x}$ jumps by $-2\mu^{-1}l k$ in the direction 
$\mathbf{F}^\perp \cap T\partial X$ and also energy between movement along magnetic line and winding redistributes thus magnetic number $j$ which has sense only away from the boundary, jumps by $O((\mu h)^{-1}k)$. 
\end{example}

Therefore even calculation of the total jump seems to be a difficult problem; most likely there is no consistent movement but just a wobbling.

Microlocal implications especially as $\mu$ is close to $h^{-1}$  seem to be really unclear. This leads to an interesting and very challenging 

\begin{Problem}\label{problem-13-7-13}
Repeat analysis of subsection~\ref{sect-13-6-3} and recover remainder estimate  $O(\mu^{-1}h^{-2})$ as $\mu h \le 1$ under appropriate assumptions:

\medskip\noindent
(i) Either magnetic drift of trajectories is larger than the possible effect of shifting

\medskip\noindent
(ii) Or movement along magnetic line is bounded by $\partial X$ only on one side and by a growing potential on the other; then jumps are going only in one direction and drift direction is disjoint from it.

\medskip\noindent
(iii) Consider Schr\"odinger-Pauli operator as $\mu h\ge 1$ and recover remainder estimate  $O(h^{-1})$ under similar assumptions.
\end{Problem}

\begin{Problem}\label{problem-13-7-14}
In the general case as $\mu h\to 0$ prove remainder estimate $o(h^{-2})$.
\end{Problem}

\bibliographystyle{alpha}

\providecommand{\bysame}{\leavevmode\hbox to3em{\hrulefill}\thinspace}

\vglue .06truein

\begin{tabular}{rrl}
&{\hskip 200 pt} &Department of Mathematics,\cr
&&University of Toronto,\cr
&&40, St.George Str.,\cr
&&Toronto, Ontario M5S 2E4\cr
&&Canada\cr
&&ivrii@math.toronto.edu\cr
&&Fax: (416)978-4107\cr
\end{tabular}

\end{document}